\documentclass{amsart}
\usepackage[all,ps,cmtip]{xy}
\usepackage{amsmath, amssymb}
\usepackage{amscd}
\usepackage{tikz,color}

\newtheorem{theorem}{Theorem}[section]
\newtheorem{lemma}[theorem]{Lemma}

\newtheorem{corollary}[theorem]{Corollary}

\newtheorem{proposition}[theorem]{Proposition}
\newtheorem{conjecture}[theorem]{Conjecture}

\newtheorem{innercustomconj}{Conjecture}
\newenvironment{customconj}[1]
  {\renewcommand\theinnercustomconj{#1}\innercustomconj}
  {\endinnercustomconj}

\newtheorem{innercustomcor}{Corollary}
\newenvironment{customcor}[1]
  {\renewcommand\theinnercustomcor{#1}\innercustomcor}
  {\endinnercustomcor}

\theoremstyle{definition}
\newtheorem{definition}[theorem]{Definition}

\theoremstyle{remark}
\newtheorem{remark}[theorem]{Remark}

\numberwithin{equation}{section}
\DeclareMathOperator{\rank}{rk}

\DeclareMathOperator{\Sing}{Sing}

\DeclareMathOperator{\codim}{codim}

\DeclareMathOperator{\Hom}{Hom}

\DeclareMathOperator{\Coh}{Coh}

\DeclareMathOperator{\D}{D}

\DeclareMathOperator{\K}{K}

\DeclareMathOperator{\NS}{NS}

\DeclareMathOperator{\ch}{ch}

\DeclareMathOperator{\Ext}{Ext}

\makeatletter \@namedef{subjclassname@2020}{\textup{2020}
Mathematics Subject Classification} \makeatother

\begin{document}
\title{Bogomolov-Gieseker type
inequalities on ruled threefolds}

\author{Hao Max Sun}
\address{Department of Mathematics, Shanghai Normal University, Shanghai 200234, People's Republic of China}

\email{hsun@shnu.edu.cn, hsunmath@gmail.com}
%\thanks{}


%    General info


\date{May 28, 2022}


\begin{abstract}
We strengthen a conjecture by the author. This conjecture is a
Bogomolov-Gieseker type inequality involving the third Chern
character of mixed tilt-stable complexes on fibred threefolds. We
extend it from complexes of mixed tilt-slope zero to arbitrary
relative tilt-slope. We show that this stronger conjecture implies
the support property of Bridgeland stability conditions, and the
existence of explicit stability conditions. We prove our conjecture
for ruled threefolds, hence improving a previous result by the
author.
\end{abstract}

\maketitle

\setcounter{tocdepth}{1}

\tableofcontents

\section{Introduction}
Throughout this paper, we let $f:\mathcal{X}\rightarrow C$ be a
projective morphism from a complex smooth projective variety of
dimension $3$ to a complex smooth projective curve such that all
scheme-theoretic fibers of $f$ are integral and normal. We denote by
$F$ the general fiber of $f$, and fix a nef and relative ample
$\mathbb{Q}$-divisor $H$ on $\mathcal{X}$. When $\mathcal{X}$ is
ruled, i.e., $\mathcal{X}=\mathbb{P}(E)$ for some rank three vector
bundle $E$ on $C$, we prove a Bogomolov-Gieseker type inequality for
the third Chern character of relative tilt-semistable objects on
$\mathcal{X}$. As a corollary, we construct a family of Bridgeland
stability conditions on $\mathcal{X}$ for which the central charge
only depends on the degrees
$$\ch_0(\star), HF\ch_1(\star), H^2\ch_1(\star), F\ch_2(\star), H\ch_2(\star), \ch_3(\star)$$
of the Chern character, and that satisfy the support property.

Stability conditions for triangulated categories were introduced by
Bridgeland in \cite{Bri1}. The existence of stability conditions on
three-dimensional varieties is often considered the biggest open
problem in the theory of Bridgeland stability conditions. In
\cite{BMT, BMS, BMSZ}, the authors introduced a conjectural
construction of Bridgeland stability conditions for any projective
threefold. Here the problem was reduced to proving a
Bogomolov-Gieseker type inequality for the third Chern character of
tilt-stable objects. It has been shown to hold for Fano 3-folds
\cite{BMSZ, Piy}, abelian 3-folds \cite{ BMS}, some product type
threefolds \cite{Kos}, quintic threefolds \cite{Li2}, threefolds
with vanishing Chern classes \cite{Sun}, etc. Yucheng Liu \cite{Liu}
showed the existence of stability conditions on product type
varieties by a different method.

In \cite{Sun2}, we give a relative version of the construction of
Bayer, Macr\`i and Toda \cite{BMT} on fibred threefolds
$\mathcal{X}$. The construction also depends on a conjectural
Bogomolov-Gieseker type inequality (\cite[Conjecture 5.2]{Sun2}) for
mixed tilt-stable complexes. In this paper, we strengthen this
conjecture from complexes of mixed tilt-slope zero to arbitrary
relative tilt-slope:

\begin{customconj}{\ref{conj1}}
Let $\mathcal{E}$ be a $\nu^{\alpha,\beta}_{H, F}$-semistable object
in $\Coh_C^{ \beta H}(\mathcal{X})$ for some $(\alpha,\beta)$ in
$\mathbb{R}_{>0}\times\mathbb{R}$ with $\nu^{\alpha,\beta}_{H,
F}(\mathcal{E})\neq\infty$. Then
\begin{eqnarray*}
&&\left(F\ch_2^{\beta}(\mathcal{E})-\frac{\alpha^2}{2}H^2F\ch_0(\mathcal{E})\right)\left(H\ch_2^{\beta}(\mathcal{E})-\frac{H^3}{3H^2F}F\ch_2^{\beta}(\mathcal{E})\right)\\
\nonumber
&\geq&\left(\ch_3^{\beta}(\mathcal{E})-\frac{\alpha^2}{2}H^2\ch_1^{\beta}(\mathcal{E})+\frac{\alpha^2H^3}{3H^2F}HF\ch_1^{\beta}(\mathcal{E})\right)HF\ch_1^{\beta}(\mathcal{E}).
\end{eqnarray*}
\end{customconj}

We show that Conjecture \ref{conj1} is equivalent to a more natural
and seemingly weaker statement (Conjecture \ref{conj2}), and it
implies the existence of explicit stability conditions on
$\mathcal{X}$ satisfying the support property if the classical
Bogomolov inequality holds on all fibers of $f$(Theorem \ref{main}).
We prove the conjecture for ruled threefolds, hence improving a
previous result by the author:

\begin{theorem}\label{thm1.1}
Conjecture \ref{conj1} holds for $\mathcal{X}=\mathbb{P}(E)$, and
there exist stability conditions satisfying the support property on
$\mathbb{P}(E)$.
\end{theorem}

As an application, we obtain the following Bogomolov type inequality
which seems stronger than the classical Bogomolov inequality.

\begin{customcor}{\ref{cor1}}
Let $\mathcal{E}$ be a $\mu_{H, F}$-semistable torsion free sheaf on
$\mathcal{X}=\mathbb{P}(E)$. Then
$$\widetilde{\Delta}_{H,F}(\mathcal{E})\geq\frac{H^3}{6H^2F}\overline{\Delta}_{H,F}(\mathcal{E})+\frac{H^3}{2H^2F}(HF\ch_1(\mathcal{E}))^2.$$
\end{customcor}

\subsection*{Organization of the paper}

Our paper is organized as follows. In Section \ref{S2}, we review
some basic results of relative slope-stability and relative
tilt-stability on a fibred variety in \cite{BLMNPS} and \cite{Sun2}.
Then in Section \ref{S3}, we recall the definition of stability
conditions, propose Conjecture \ref{conj1} and \ref{conj2} and give
a conjectural construction of Bridgeland stability conditions on
fibred threefolds. We prove Theorem \ref{thm1.1} and Corollary
\ref{cor1} in Section \ref{S4}.

\subsection*{Notation}
Let $X$ be a smooth projective variety. We denote by $T_X$ and
$\Omega_X^1$ the tangent bundle and cotangent bundle of $X$,
respectively. $K_X$ and $\omega_X$ denote the canonical divisor and
canonical sheaf of $X$, respectively. We write $c_i(X):=c_i(T_X)$
for the $i$-th Chern class of $X$. We write $\NS(X)$ for the
N\'eron-Severi group of divisors up to numerical equivalence. We
also write $\NS(X)_{\mathbb{Q}}$, $\NS(X)_{\mathbb{R}}$, etc. for
$\NS(X)\otimes\mathbb{Q}$, etc. For a triangulated category
$\mathcal{D}$, we write $\K(\mathcal{D})$ for its Grothendieck
group. For a variety $Y$, we denote by $\Sing Y$ the singular locus
of $Y$.

Let $\pi:\mathcal{X}\rightarrow S$ be a flat morphism of Noetherian
schemes and $W\subset S$ be a subscheme. We denote by
$\mathcal{X}_W=\mathcal{X}\times_S W$ the fiber of $\pi$ over $W$,
and by $i_W:\mathcal{X}_W\hookrightarrow \mathcal{X}$ the embedding
of the fiber. In the case that $S$ is integral, we write $K(S)$ for
its fraction field, and $\mathcal{X}_{K(S)}$ for the generic fiber
of $\pi$. We denote by $\D^b(\mathcal{X})$ the bounded derived
category of coherent sheaves on $\mathcal{X}$. Given $E\in
\D^b(\mathcal{X})$, we write $E_W$ (resp., $E_{K(S)}$) for the
pullback to $\mathcal{X}_W$ (resp., $\mathcal{X}_{K(S)}$).

Let $F$ be a coherent sheaf on $X$. We write $H^j(F)$ ($j\in
\mathbb{Z}_{\geq0}$) for the cohomology groups of $F$ and write
$\dim F$ for the dimension of its support. We write $\Coh_{\leq
d}(X)\subset\Coh(X)$ for the subcategory of sheaves supported in
dimension $\leq d$. Given a bounded t-structure on $\D^b(X)$ with
heart $\mathcal{A}$ and an object $E\in \D^b(X)$, we write
$\mathcal{H}_{\mathcal{A}}^j(E)$ ($j\in \mathbb{Z}$) for the
cohomology objects with respect to $\mathcal{A}$. When
$\mathcal{A}=\Coh(X)$, we simply write $\mathcal{H}^j(E)$. Given a
complex number $z\in\mathbb{C}$, we denote its real and imaginary
part by $\Re z$ and $\Im z$, respectively. We write
$\sqrt{\mathbb{Q}_{>0}}$ for the set $\{\sqrt{x}: x\in
\mathbb{Q}_{>0}\}$.

%Numerical equivalence of two
%divisors $D_1$, $D_2$ on $X$ is denoted by $D_1\sim_{num} D_2$.

\subsection*{Acknowledgments}
The author was supported by National Natural Science
Foundation of China (Grant No. 11771294, 11301201).

\section{Relative tilt-stability}\label{S2}
We will review some basic results on the relative slope-stability
and relative tilt-stability in \cite{BLMNPS} and \cite{Sun2}.

\subsection{Stability for sheaves}
For any $\mathbb{R}$-divisor $D$ on $\mathcal{X}$, we define the
twisted Chern character $\ch^{D}=e^{-D}\ch$. More explicitly, we
have
\begin{eqnarray*}
\begin{array}{lcl}
\ch^{D}_0=\ch_0=\rank  && \ch^{D}_2=\ch_2-D\ch_1+\frac{D^2}{2}\ch_0\\
&&\\
\ch^{D}_1=\ch_1-D\ch_0 &&
\ch^{D}_3=\ch_3-D\ch_2+\frac{D^2}{2}\ch_1-\frac{D^3}{6}\ch_0.
\end{array}
\end{eqnarray*}

We define the relative slope $\mu_{H, F}$ of a coherent sheaf
$\mathcal{E}\in \Coh(\mathcal{X})$ by
\begin{eqnarray*}
\mu_{H, F}(\mathcal{E})= \left\{
\begin{array}{lcl}
+\infty,  & &\mbox{if}~\ch_0(\mathcal{E})=0,\\
&&\\
\frac{FH\ch_1(\mathcal{E})}{FH^{2}\ch_0(\mathcal{E})}, &
&\mbox{otherwise}.
\end{array}\right.
\end{eqnarray*}

\begin{definition}\label{slope}
A coherent sheaf $\mathcal{E}$ on $\mathcal{X}$ is $\mu_{H,
F}$-(semi)stable (or relative slope-(semi)stable) if, for all
non-zero subsheaves $\mathcal{F}\hookrightarrow \mathcal{E}$, we
have
$$\mu_{H, F}(\mathcal{F})<(\leq)\mu_{H, F}(\mathcal{E}/\mathcal{F}).$$
\end{definition}

Similarly, for any point $s\in C$, we can define
$\mu_{H_s}$-stability (or slope-stability) of a coherent sheaf
$\mathcal{G}$ on the fiber $\mathcal{X}_s$ over $s$ for the slope
$\mu_{H_s}$:
\begin{eqnarray*}
\mu_{H_s}(\mathcal{G})= \left\{
\begin{array}{lcl}
+\infty,  & &\mbox{if}~\ch_0(\mathcal{G})=0,\\
&&\\
\frac{H_s\ch_1(\mathcal{G})}{H_s^{2}\ch_0(\mathcal{G})}, &
&\mbox{otherwise}.
\end{array}\right.
\end{eqnarray*}
Here $\ch_1(\mathcal{G})$ is defined as a Weil divisor up to linear
equivalence such that $$\ch_1(\mathcal{G})|_{\mathcal{X}_s-\Sing
\mathcal{X}_s}=\ch_1(\mathcal{G}|_{\mathcal{X}_s-\Sing
\mathcal{X}_s}).$$ Since, by our assumption, $\codim (\Sing
\mathcal{X}_s)\geq2$, $\ch_1(\mathcal{G})$ is well-defined. One sees
that $\mu_{H, F}(\mathcal{E})=\mu_{H_s}(\mathcal{E}_s)$.

The below lemma gives the relation between the Chern characters of
objects on fibers and their pushforwards.

\begin{lemma}\label{Chern}
Let $W$ be a closed subscheme of $C$, and $j$ be a positive integer.
\begin{enumerate}
\item For any $\mathcal{E}\in\D^b(\mathcal{X})$ and
$\mathbb{Q}$-divisor $D$ on $\mathcal{X}$, we have
$$\ch_j^D(i_{W*}\mathcal{E}_W)=\mathcal{X}_W\ch^D_{j-1}(\mathcal{E}).$$
\item Assume that $W$ is a closed point of $C$. Then for any
$\mathcal{Q}\in\D^b(\mathcal{X}_W)$ we have
$$\ch_{j}(i_{W*}\mathcal{Q})=i_{W*}\ch_{j-1}(\mathcal{Q})$$ if $0\leq
j\leq2$. Moreover, we have
$\ch_{3}(i_{W*}\mathcal{Q})=i_{W*}\ch_{2}(\mathcal{Q})$ if
$\mathcal{X}_W$ is smooth.
\end{enumerate}
\end{lemma}
\begin{proof}
See \cite[Lemma 2.7]{Sun2} for part (1) and Part (2) in the case
that $\mathcal{X}_W$ is smooth. Part (2) for the case that
$\mathcal{X}_W$ is singular follows from applying the
Grothendieck-Riemann-Roch theorem for the embedding
$$i_W|_{\mathcal{X}_W-\Sing\mathcal{X}_W}:\mathcal{X}_W-\Sing\mathcal{X}_W\hookrightarrow
\mathcal{X}.$$
\end{proof}

\begin{definition}\label{tor}
Let $\mathcal{A}_C$ be the heart of a $C$-local $t$-structure on
$\D^b(\mathcal{X})$ (see \cite[Definition 4.10]{BLMNPS}), and let
$\mathcal{E}\in \mathcal{A}_C$.
\begin{enumerate}
\item We say $\mathcal{E}$ is $C$-flat if $\mathcal{E}_c\in \mathcal{A}_c$
for every point $c\in C$, where $\mathcal{A}_c$ is the heart of the
$t$-structure given by \cite[Theorem 5.3]{BLMNPS} applied to the
embedding $c\hookrightarrow C$.

\item An object $\mathcal{F}\in\D^b(\mathcal{X})$ is called $C$-torsion if
it is the pushforward of an object in $\D^b(\mathcal{X}_W)$ for some
proper closed subscheme $W\subset C$.

\item $\mathcal{E}$ is called $C$-torsion free if it
contains no nonzero $C$-torsion subobject.
\end{enumerate}
We denote by $\mathcal{A}_{C\text{-tor}}$ the subcategory of
$C$-torsion objects in $\mathcal{A}_C$, and by
$\mathcal{A}_{C\text{-tf}}$ the subcategory of $C$-torsion free
objects. We say $\mathcal{A}_C$ has a $C$-torsion theory if the pair
of subcategories $(\mathcal{A}_{C\text{-tor}},
\mathcal{A}_{C\text{-tf}})$ forms a torsion pair in the sense of
\cite[Definition 4.6]{BLMNPS}.
\end{definition}

\begin{lemma}\label{flat}
Let $\mathcal{E}\in \mathcal{A}_C$ be as in Definition \ref{tor}.
Then
\begin{enumerate}
\item $\mathcal{E}$ is $C$-flat if and only if $\mathcal{E}$ is
$C$-torsion free;
\item $\mathcal{E}$ is $C$-torsion if and only if
$\mathcal{E}_{K(C)}=0$.
\end{enumerate}
\end{lemma}
\begin{proof}
See \cite[Lemma 6.12 and Lemma 6.4]{BLMNPS}.
\end{proof}

Since  $\Coh(\mathcal{X})$ is the heart of the natural $C$-local
$t$-structure on $\D^b(\mathcal{X})$, one can applies the above
definition and lemma to coherent sheaves. The following lemma shows
the relation of the relative slope-stability to the slop-stability.

\begin{lemma}\label{f-stable}
Let $\mathcal{E}$ be a $C$-torsion free sheaf on $\mathcal{X}$. Then
$\mathcal{E}$ is $\mu_{H, F}$-(semi)stable if and only if there
exists an open subset $U\subset C$ such that $\mathcal{E}_{s}$ is
$\mu_{H_s}$-(semi)stable for any point $s\in U$.
\end{lemma}
\begin{proof}
See \cite[Lemma 2.4]{Sun2}.
\end{proof}

We recall the classical Bogomolov inequality:
\begin{theorem}\label{Bog}
Assume that $\mathcal{E}$ is a $\mu_{H,F}$-semistable torsion free
sheaf on $\mathcal{X}$. Then we have
\begin{eqnarray*}
F\Delta(\mathcal{E})&:=&F\big(\ch_1^2(\mathcal{E})-2\ch_0(\mathcal{E})\ch_2(\mathcal{E})\big)\geq0;\\
H\Delta(\mathcal{E})&:=&H\big(\ch_1^2(\mathcal{E})-2\ch_0(\mathcal{E})\ch_2(\mathcal{E})\big)\geq0.
\end{eqnarray*}
\end{theorem}
\begin{proof}
See \cite[Theorem 3.2]{Langer1}.
\end{proof}

Let $\beta$ be a real number. For brevity, we write $\ch^{\beta}$
for the twisted Chern character $\ch^{\beta H}$. A short calculation
shows
\begin{eqnarray*}
\Delta(\mathcal{E})&:=&(\ch_1(\mathcal{E}))^2-2\ch_0(\mathcal{E})\ch_2(\mathcal{E})\\
&=&(\ch^{\beta}_1(\mathcal{E}))^2-2\ch^{\beta}_0(\mathcal{E})\ch^{\beta}_2(\mathcal{E}).
\end{eqnarray*}

\begin{definition}
We define the generalized relative discriminants
$$\overline{\Delta}^{\beta
H}_{H,F}:=(HF\ch^{\beta}_1)^2-2H^{2}F\ch^{\beta}_0\cdot(F\ch^{\beta}_2)$$
and $$\widetilde{\Delta}^{\beta
H}_{H,F}:=(HF\ch_1^{\beta})(H^{2}\ch_1^{\beta})-H^{2}F\ch^{\beta}_0\cdot(H\ch^{\beta}_2).$$
\end{definition}
A short calculation shows $$\overline{\Delta}^{\beta
H}_{H,F}=(HF\ch_1)^2-2H^{2}F\ch_0\cdot(F\ch_2)=\overline{\Delta}_{H,F}.$$
Hence the first generalized relative discriminant
$\overline{\Delta}^{\beta H}_{H,F}$ is independent of $\beta$. In
general $\widetilde{\Delta}^{\beta H}_{H,F}$ is not independent of
$\beta$, but we have $$\widetilde{\Delta}^{\beta
H}_{H,F}(\mathcal{E}\otimes\mathcal{O}_{\mathcal{X}}(mF))=\widetilde{\Delta}^{\beta
H}_{H,F}(\mathcal{E}),$$ for any $\mathcal{E}\in\D^b(\mathcal{X})$
and $m\in\mathbb{Z}$.

\begin{theorem}\label{Bog1}
Assume that $\mathcal{E}$ is a $\mu_{H,F}$-semistable torsion free
sheaf on $\mathcal{X}$. Then we have $\overline{\Delta}^{\beta
H}_{H,F}(\mathcal{E})\geq0$ and $\widetilde{\Delta}^{\beta
H}_{H,F}(\mathcal{E})\geq0$.
\end{theorem}
\begin{proof}
See \cite[Theorem 3.10]{Sun2}.
\end{proof}

\begin{definition}
Let $p$ be a point of $C$. We say the Bogomolov inequality holds on
the fiber $\mathcal{X}_p$, if
$$(H\ch_2(i_{p*}E))^2\geq2H^2\ch_1(i_{p*}E))\ch_3(i_{p*}E))$$ for any
$\mu_{H_p}$-semistable sheaf $E\in \Coh(\mathcal{X}_p)$.
\end{definition}
A short computation shows
\begin{eqnarray*}
&&(H\ch_2(i_{p*}E))^2-2H^2\ch_1(i_{p*}E))\ch_3(i_{p*}E))\\
&=&(H\ch^{\beta}_2(i_{p*}E))^2-2H^2\ch^{\beta}_1(i_{p*}E))\ch^{\beta}_3(i_{p*}E)).
\end{eqnarray*}
By Lemma \ref{Chern}, one sees
\begin{eqnarray*}
&&(H\ch_2(i_{p*}E))^2-2H^2\ch_1(i_{p*}E))\ch_3(i_{p*}E))\\
&=&(H_p\ch_{\mathcal{X}_p, 1}(E))^2-2H^2_p\ch_{\mathcal{X}_p,
0}(E)\ch_{\mathcal{X}_p, 2}(E)
\end{eqnarray*}
if $\mathcal{X}_p$ is smooth. Hence the Bogomolov inequality holds
on every smooth fiber of $f$.

Another important notion of stability for a sheaf on a fibration is
the stability introduced in \cite[Example 15.3]{BLMNPS}. We define
the slope $\mu_{C}$ of a coherent sheaf $\mathcal{E}\in
\Coh(\mathcal{X})$ by

\begin{eqnarray*}
\mu_{C}(\mathcal{E})= \left\{
\begin{array}{ll}
\frac{HF\ch_1(\mathcal{E})}{H^{2}F\ch_0(\mathcal{E})}, &\mbox{if}~\ch_0(\mathcal{E})\neq0,\\
\frac{H\ch_2(\mathcal{E})}{H^{2}\ch_1(\mathcal{E})},&\mbox{if}~\mathcal{E}_{K(C)}=0~\mbox{and}~H^{2}\ch_1(\mathcal{E})\neq0,\\
+\infty, &\mbox{otherwise}.
\end{array}\right.
\end{eqnarray*}
We can define $\mu_C$-stability as in Definition \ref{slope}.

\begin{definition}
A coherent sheaf $\mathcal{E}$ on $\mathcal{X}$ is
$\mu_{C}$-(semi)stable if, for all non-zero subsheaves
$\mathcal{F}\hookrightarrow \mathcal{E}$, we have
$$\mu_{C}(\mathcal{F})<(\leq)\mu_{C}(\mathcal{E}/\mathcal{F}).$$
\end{definition}

Like the slope-stability and the relative slope-stability, the
$\mu_{C}$-stability also satisfies the following weak see-saw
property and the Harder-Narasimhan property (see \cite[Proposition
16.6]{BLMNPS} and \cite[Proposition 2.6]{Sun2}).
\begin{proposition}
Let $\mathcal{E}\in\Coh(\mathcal{X})$ be a non-zero sheaf.
\begin{enumerate}
\item For any short exact sequence $$0\rightarrow \mathcal{F}\rightarrow \mathcal{E}\rightarrow
\mathcal{G}\rightarrow0$$ in $\Coh(\mathcal{X})$, we have
$$\mu_{C}(\mathcal{F})\leq\mu_{C}(\mathcal{E})\leq\mu_{C}(\mathcal{G})~\mbox{or}~\mu_{C}(\mathcal{F})\geq\mu_{C}(\mathcal{E})\geq\mu_{C}(\mathcal{G}).$$

\item There is a filtration (called Harder-Narasimhan filtration)
$$0=\mathcal{E}_0\subset \mathcal{E}_1\subset\cdots\subset \mathcal{E}_m=\mathcal{E}$$
such that: $\mathcal{G}_i:=\mathcal{E}_i/\mathcal{E}_{i-1}$ is
$\mu_{C}$-semistable, and
$\mu_{C}(\mathcal{G}_1)>\cdots>\mu_{C}(\mathcal{G}_m)$. We write
$\mu^+_{C}(\mathcal{E}):=\mu_{C}(\mathcal{G}_1)$ and
$\mu^-_{C}(\mathcal{E}):=\mu_{C}(\mathcal{G}_m)$.
\end{enumerate}
\end{proposition}

Lemma \ref{f-stable} says that the usual notion of relative
slope-stability for a torsion free sheaf is equivalent to the
slope-stability of the general fiber of the sheaf. In contrast,
$\mu_C$-stability requires stability for all fibers:
\begin{proposition}\label{pro2.11}
Let $\mathcal{E}$ be a $C$-torsion free sheaf on $\mathcal{X}$. Then
$\mathcal{E}$ is $\mu_C$-semistable if and only if $\mathcal{E}$ is
$\mu_{H,F}$-semistable and for any closed point $p\in C$ and any
quotient $\mathcal{E}_p\twoheadrightarrow \mathcal{Q}$ in
$\Coh(\mathcal{X}_p)$ we have
$\mu_{H_p}(\mathcal{E}_p)\leq\mu_{H_p}(\mathcal{Q})$.
\end{proposition}
\begin{proof}
See \cite[Lemma 15.7]{BLMNPS}.
\end{proof}

\subsection{Relative tilt-stability} Let $\beta$ be a real
number and $\alpha$ be a positive real number. There exists a
\emph{torsion pair} $(\mathcal{T}_{\beta H},\mathcal{F}_{\beta H})$
in $\Coh(\mathcal{X})$ defined as follows:
\begin{eqnarray*}
\mathcal{T}_{\beta H}&=&\{\mathcal{E}\in\Coh(\mathcal{X}):\mu^-_{C}(\mathcal{E})>\beta \}\\
\mathcal{F}_{\beta
H}&=&\{\mathcal{E}\in\Coh(\mathcal{X}):\mu^+_{C}(\mathcal{E})\leq\beta
\}.
\end{eqnarray*}
Equivalently, $\mathcal{T}_{\beta H}$ and $\mathcal{F}_{\beta H}$
are the extension-closed subcategories of $\Coh(\mathcal{X})$
generated by $\mu_{C}$-stable sheaves with $\mu_{C}$-slope $>\beta$
and $\leq\beta$, respectively.

\begin{definition}\label{def3.3}
We let $\Coh_C^{\beta H}(\mathcal{X})\subset \D^b(\mathcal{X})$ be
the extension-closure
$$\Coh_C^{\beta H}(\mathcal{X})=\langle\mathcal{T}_{\beta H}, \mathcal{F}_{
\beta H}[1]\rangle.$$
\end{definition}

By the general theory of torsion pairs and tilting \cite{HRS},
$\Coh_C^{\beta H}(\mathcal{X})$ is the heart of a bounded
t-structure on $\D^b(\mathcal{X})$; in particular, it is an abelian
category. For any point $s\in C$ such that $\mathcal{X}_s$ is
smooth, similar as Definition \ref{def3.3}, one can define the
subcategory
$$\Coh^{\beta H_s}(\mathcal{X}_s)=\langle\mathcal{T}_{\beta H_s}, \mathcal{F}_{
\beta H_s}[1]\rangle\subset\D^b(\mathcal{X}_s)$$  via the
$\mu_{H_s}$-stability (see \cite[Section 14.2]{BLMNPS}). For any
$\mathcal{E}\in \Coh_C^{ \beta H}(\mathcal{X})$, its relative
tilt-slope $\nu_{H,F}^{\alpha, \beta}$ is defined by
\begin{eqnarray*}
\nu_{H,F}^{\alpha, \beta}(\mathcal{E})= \left\{
\begin{array}{lcl}
+\infty,  & &\mbox{if}~FH\ch^{\beta}_1(\mathcal{E})=0,\\
&&\\
\frac{F\ch_2^{\beta}(\mathcal{E})-\frac{1}{2}\alpha^2FH^{2}\ch^{\beta}_0(\mathcal{E})}{HF\ch^{\beta}_1(\mathcal{E})},
& &\mbox{otherwise}.
\end{array}\right.
\end{eqnarray*}

\begin{definition}
An object $\mathcal{E}\in\Coh_C^{ \beta H}(\mathcal{X})$ is
$\nu_{H,F}^{\alpha,\beta}$-(semi)stable (or relative
tilt-(semi)stable) if, for all non-zero subobjects
$\mathcal{F}\hookrightarrow \mathcal{E}$, we have
$$\nu_{H,F}^{\alpha,\beta}(\mathcal{F})<(\leq)\nu_{H,F}^{\alpha,\beta}(\mathcal{E}/\mathcal{F}).$$
\end{definition}
We can also consider the tilt-stability on the smooth fibers of $f$.
For any point $s\in C$ such that $\mathcal{X}_s$ is smooth, the
tilt-slope $\nu_{s}^{\alpha, \beta}$ of an object $\mathcal{G}\in
\Coh^{ \beta H_s}(\mathcal{X}_s)$ is defined by
\begin{eqnarray*}
\nu_{s}^{\alpha, \beta}(\mathcal{G})= \left\{
\begin{array}{lcl}
+\infty,  & &\mbox{if}~H_s\ch_{\mathcal{X}_s,1}^{\beta}(\mathcal{G})=0,\\
&&\\
\frac{\ch_{\mathcal{X}_s,2}^{\beta}(\mathcal{G})-\frac{1}{2}\alpha^2H_s^{2}\ch_{\mathcal{X}_s,0}^{\beta}(\mathcal{G})}{H_s\ch_{\mathcal{X}_s,1}^{\beta}(\mathcal{G})},
& &\mbox{otherwise}.
\end{array}\right.
\end{eqnarray*}
This gives the tilt-stability condition on $\mathcal{X}_s$ defined
in \cite{BMT, BMS}. One sees
\begin{eqnarray*}
\nu_s^{\alpha,
\beta}(\mathcal{E}_s)=\nu_{H,F}^{\alpha,\beta}(\mathcal{E}).
\end{eqnarray*}

Relative tilt-stability gives a notion of stability, in the sense
that Harder-Narasimhan filtrations exist:
\begin{lemma}\label{HN}
Let $\mathcal{E}$ be an object in $\Coh_C^{ \beta H}(\mathcal{X})$.
Then there is a filtration
$$0=\mathcal{E}_0\subset \mathcal{E}_1\subset\cdots\subset \mathcal{E}_m=\mathcal{E}$$
such that: $\mathcal{G}_i:=\mathcal{E}_i/\mathcal{E}_{i-1}$ is
$\nu_{H,F}^{\alpha,\beta}$-semistable, and
$\nu_{H,F}^{\alpha,\beta}(\mathcal{G}_1)>\cdots>\nu_{H,F}^{\alpha,\beta}(\mathcal{G}_m)$.
We write
$\nu_{H,F}^{\alpha,\beta,+}(\mathcal{E}):=\nu_{H,F}^{\alpha,\beta}(\mathcal{G}_1)$
and
$\nu_{H,F}^{\alpha,\beta,-}(\mathcal{E}):=\nu_{H,F}^{\alpha,\beta}(\mathcal{G}_m)$.
\end{lemma}
\begin{proof}
Since one has the standard exact sequence
$\mathcal{H}^{-1}(\mathcal{E})[1]\hookrightarrow\mathcal{E}\twoheadrightarrow\mathcal{H}^{0}(\mathcal{E})$
in $\Coh_C^{ \beta H}(\mathcal{X})$, by \cite[Lemma 6.17]{BLMNPS},
one sees that $\mathcal{E}$ admit a unique maximal $C$-torsion
subobject, namely $\Coh_C^{ \beta H}(\mathcal{X})$ has a $C$-torsion
theory. We denote by $\mathcal{E}_{C\text{-tor}}$ the maximal
$C$-torsion subobject of $\mathcal{E}$ and
$\mathcal{E}_{C\text{-tf}}=\mathcal{E}/\mathcal{E}_{C\text{-tor}}$
the $C$-torsion free quotient of $\mathcal{E}$. It will be enough to
show that the Harder-Narasimhan filtration exists for
$\mathcal{E}_{C\text{-tf}}$.

By \cite[Appendix 2]{BMS}, one sees that the Harder-Narasimhan
filtration exists for $(\mathcal{E}_{C\text{-tf}})_{K(C)}$ with
respect to $\nu_{K(C)}^{\alpha,\beta}$-stability. We denote it by
$$0=\widetilde{\mathcal{E}}_0\subset
\widetilde{\mathcal{E}}_1\subset\cdots\subset\widetilde{\mathcal{E}}_m=(\mathcal{E}_{C\text{-tf}})_{K(C)}.$$
One can lift it to a filtration in $\Coh_C^{ \beta H}(\mathcal{X})$
$$0=\mathcal{E}_0\subset \mathcal{E}_1\subset\cdots\subset \mathcal{E}_m=\mathcal{E}_{C\text{-tf}}$$
from \cite[Lemma 4.16.(3)]{BLMNPS}. Replacing $\mathcal{E}_j$ by the
kernel of
$\mathcal{E}_{j+1}\rightarrow(\mathcal{E}_{j+1}/\mathcal{E}_{j})_{C\text{-tf}}$,
one can assume that $\mathcal{E}_{j+1}/\mathcal{E}_{j}$ is
$C$-torsion free. Hence $\mathcal{E}_{j+1}/\mathcal{E}_{j}$ is
$\nu_{H,F}^{\alpha,\beta}$-semistable by the following lemma and
$$\nu_{H,F}^{\alpha,\beta}(\mathcal{E}_{j}/\mathcal{E}_{j-1})=\nu_{K(C)}^{\alpha,\beta}(\widetilde{\mathcal{E}}_{j}/\widetilde{\mathcal{E}}_{j-1})>
\nu_{K(C)}^{\alpha,\beta}(\widetilde{\mathcal{E}}_{j+1}/\widetilde{\mathcal{E}}_{j})=\nu_{H,F}^{\alpha,\beta}(\mathcal{E}_{j+1}/\mathcal{E}_{j}).$$
This completes the proof.
\end{proof}

\begin{lemma}\label{K-stable}
Let $\mathcal{E}\in \Coh_C^{ \beta H}(\mathcal{X})$ be a $C$-torsion
free object. Then the following conditions are equivalent:
\begin{enumerate}
\item $\mathcal{E}$ is $\nu_{H,F}^{\alpha, \beta}$-(semi)stable;

\item $\mathcal{E}_{K(C)}$ is $\nu_{K(C)}^{\alpha,
\beta}$-(semi)stable.
\end{enumerate}
\end{lemma}
\begin{proof}
The proof is the same as that of \cite[Lemma 3.15]{Sun2}.
\end{proof}

\begin{proposition}\label{Noether}
Assume that $\beta\in \mathbb{Q}$. Then the category $\Coh_C^{ \beta
H}(\mathcal{X})$ is noetherian.
\end{proposition}
\begin{proof}
The conclusion was showed in the proof of \cite[Theorem 3.11]{Sun2}
in the case that $f$ is smooth.

For our case, let
$\mathcal{E}_1\twoheadrightarrow\mathcal{E}_2\twoheadrightarrow\cdots$
be an infinite sequence of surjections in $\Coh_C^{ \beta
H}(\mathcal{X})$. By the proof of Lemma \ref{HN}, one sees that
$\Coh_C^{ \beta H}(\mathcal{X})$ has a $C$-torsion theory. Hence the
induced sequence of surjections
$(\mathcal{E}_1)_{K(C)}\twoheadrightarrow(\mathcal{E}_2)_{K(C)}\twoheadrightarrow\cdots$
stabilizes by the proof of \cite[Lemma 3.2.4]{BMT}, in other words
we may assume that the kernel $\mathcal{K}_i$ of every surjection
$\mathcal{E}_1\twoheadrightarrow\mathcal{E}_{i}$ is $C$-torsion.
Then we obtains a chain of injections
\begin{equation}\label{2.1}
0=\mathcal{K}_1\subset\mathcal{K}_2\subset\cdots\subset\mathcal{E}_{1,C\text{-tor}},
\end{equation}
where $\mathcal{E}_{1, C\text{-tor}}$ is the maximal $C$-torsion
subobject of $\mathcal{E}_1$. It induces the chain of injections
$$0\subset\mathcal{H}^{-1}(\mathcal{K}_2)\subset\mathcal{H}^{-1}(\mathcal{K}_3)\subset\cdots\subset\mathcal{H}^{-1}(\mathcal{E}_{1,C\text{-tor}}).$$
Thus we can assume that
$\mathcal{H}^{-1}(\mathcal{K}_2)=\mathcal{H}^{-1}(\mathcal{K}_i)$
for $i\geq2$ as $\Coh(\mathcal{X})$ is noetherian.

Let us consider the sequence of surjections
$$\mathcal{E}_{1,C\text{-tor}}\twoheadrightarrow\mathcal{E}_{1,C\text{-tor}}/\mathcal{K}_2\twoheadrightarrow\cdots
\twoheadrightarrow\mathcal{E}_{1,C\text{-tor}}/\mathcal{K}_i\twoheadrightarrow\cdots.$$
Then we have the chain of surjections
$$\mathcal{H}^0(\mathcal{E}_{1,C\text{-tor}})\twoheadrightarrow\mathcal{H}^0(\mathcal{E}_{1,C\text{-tor}}/\mathcal{K}_2)\twoheadrightarrow\cdots\twoheadrightarrow
\mathcal{H}^0(\mathcal{E}_{1,C\text{-tor}}/\mathcal{K}_i)\twoheadrightarrow.$$
So we may assume that
$\mathcal{H}^0(\mathcal{E}_{1,C\text{-tor}})=\mathcal{H}^0(\mathcal{E}_{1,C\text{-tor}}/\mathcal{K}_i)$
for $i\geq1$. By the construction of $\Coh_C^{ \beta
H}(\mathcal{X})$, one sees
$H\ch^{\beta}_2(\mathcal{E}_{1,C\text{-tor}}/\mathcal{K}_i)\geq0$.
Thus we may also assume that
$H\ch^{\beta}_2(\mathcal{E}_{1,C\text{-tor}}/\mathcal{K}_i)$ is
independent of $i$ from the discreteness of $H\ch^{\beta}_2$. This
implies that $H\ch^{\beta}_2(\mathcal{K}_i)=0$ for any $i\geq1$. By
the following lemma, one sees $\mathcal{H}^0(\mathcal{K}_i)$ is a
torsion sheaf supported in dimension zero. By setting
$V=\mathcal{H}^{-1}(\mathcal{E}_{1,C\text{-tor}})/\mathcal{H}^{-1}(\mathcal{K}_i)$,
we have the exact sequence
$$0\rightarrow
V\rightarrow\mathcal{H}^{-1}(\mathcal{E}_{1,C\text{-tor}}/\mathcal{K}_i)\rightarrow\mathcal{H}^0(\mathcal{K}_i)\rightarrow0.$$
Let us assume without loss of generality that $V=i_{p*}V^{\prime}$,
$\mathcal{H}^{-1}(\mathcal{E}_{1,C\text{-tor}}/\mathcal{K}_i)=i_{p*}E_i^{\prime}$
and $\mathcal{H}^0(\mathcal{K}_i)=i_{p*}Q_i^{\prime}$ for
$V^{\prime}, E_i^{\prime}, Q_i^{\prime}\in\Coh(\mathcal{X}_p)$. By
the construction of $\Coh_C^{ \beta H}(\mathcal{X})$, one sees that
$E_i^{\prime}$ is torsion free. So is $V^{\prime}$. The integrality
of $\mathcal{X}_p$ implies that $V^{\prime}$ is a subsheaf of
$$(V^{\prime})^{\vee\vee}:=\mathcal{H}om(\mathcal{H}om(V^{\prime},\mathcal{O}_{\mathcal{X}_p}),
\mathcal{O}_{\mathcal{X}_p}).$$ By the normality of $\mathcal{X}_p$,
one sees that $(V^{\prime})^{\vee\vee}$ is reflexive and the
quotient $(V^{\prime})^{\vee\vee}/V^{\prime}$ is supported on
$\mathcal{X}_p$ in dimension zero. It follows that $Q_i^{\prime}$ is
a subsheaf of $(V^{\prime})^{\vee\vee}/V^{\prime}$. In particular
the length of $\mathcal{H}^0(\mathcal{K}_i)$ is bounded. Therefore
$\mathcal{H}^0(\mathcal{K}_i)=\mathcal{H}^0(\mathcal{K}_{i+1})$ for
large $i$, and the sequence (\ref{2.1}) terminates. This completes
the proof.
\end{proof}

\begin{lemma}\label{ch2}
Let $\mathcal{E}$ be an object in $\Coh_C^{\beta H}(\mathcal{X})$.
\begin{enumerate}
\item We have $HF\ch^{\beta}_1(\mathcal{E})\geq0$.
\item If $HF\ch^{\beta}_1(\mathcal{E})=0$, then one has
$H\ch^{\beta}_2(\mathcal{E})\geq0$,
$F\ch^{\beta}_2(\mathcal{E})\geq0$ and $\ch_0(\mathcal{E})\leq0$.
\item If $HF\ch^{\beta}_1(\mathcal{E})=\ch_0(\mathcal{E})=H\ch^{\beta}_2(\mathcal{E})=0$, then
$\mathcal{H}^0(\mathcal{E})\in\Coh_{\leq 0}(\mathcal{X})$,
$\mathcal{H}^{-1}(\mathcal{E})$ is a $C$-torsion $\mu_C$-semistable
sheaf with $\mu_C(\mathcal{H}^{-1}(\mathcal{E}))=\beta$. Moreover,
if the Bogomolov inequality holds on every fiber of $f$, then
$\ch^{\beta}_3(\mathcal{E})\geq0$.
\end{enumerate}
\end{lemma}
\begin{proof}
The proof is the same as that of \cite[Lemma 3.14]{Sun2}.
\end{proof}

\subsection{Properties of relative tilt-stability}
The relative tilt-stability is unchanged under tensoring with line
bundles $\mathcal{O}_{\mathcal{X}}(aF)$ for any $a\in\mathbb{Z}$:
\begin{lemma}\label{tensor}
If $\mathcal{E}\in\Coh_C^{\beta H}(\mathcal{X})$ is
$\nu_{H,F}^{\alpha,\beta}$-(semi)stable, then
$\mathcal{E}(aF)\in\Coh_C^{\beta H}(\mathcal{X})$ and
$\mathcal{E}(aF)$ is also $\nu_{H,F}^{\alpha,\beta}$-(semi)stable
for any $a\in \mathbb{Z}$.
\end{lemma}
\begin{proof}
The conclusion follows directly from the definition of
$\Coh_C^{\beta H}(\mathcal{X})$ and the equality
$\nu_{H,F}^{\alpha,\beta}(\mathcal{E})=\nu_{H,F}^{\alpha,\beta}(\mathcal{E}(aF))$
for any $a\in\mathbb{Z}$.
\end{proof}

The Bogomolov type inequality holds for relative tilt-stable
complexes:

\begin{theorem}\label{Bog-tilt}
If $\mathcal{E}\in\Coh_C^{\beta H}(\mathcal{X})$ is
$\nu_{H,F}^{\alpha,\beta}$-semistable, then
$$\overline{\Delta}_{H,F}(\mathcal{E})\geq0.$$
\end{theorem}
\begin{proof}
The result follows from \cite[Theorem 3.5]{BMS} and Lemma
\ref{K-stable}.
\end{proof}

Let $v:\K(\D^b(\mathcal{X}))\rightarrow\Lambda:=
\mathbb{Z}\oplus\mathbb{Z}\oplus\frac{1}{2}\mathbb{Z}$ be the map
given by $$v(\mathcal{E})=(H^2F\ch_0(\mathcal{E}),
HF\ch_1(\mathcal{E}), F\ch_2(\mathcal{E})).$$ Notice that
$\nu_{H,F}^{\alpha,\beta}$ factors through $v$. A numerical wall in
relative tilt-stability with respect to a class $u\in \Lambda$ is a
non trivial proper subset $W$ of the $(\alpha, \beta)$-plane
$\mathbb{R}_{>0}\times\mathbb{R}$ given by an equation of the form
$\nu_{H,F}^{\alpha,\beta}(u)=\nu_{H,F}^{\alpha,\beta}(w)$ for
another class $w\in \Lambda$. By Lemma \ref{K-stable}, one can
translate some basic properties of the structure of walls in
tilt-stability into relative tilt-stability. Part (1)-(3) is usually
called ``Bertram's Nested Wall Theorem'' and proved in \cite{Bri2},
\cite{Maci} and \cite[Lemma 4.3]{BMS}, while part (4), (5), (6) and
(7) are in Appendix 1 and Lemma 2.7 of \cite{BMS}.

\begin{proposition}\label{Wall}
Let $\mathcal{E}$ be an object in $\D^b(\mathcal{X})$ and $u\in
\Lambda$ a fixed class.
\begin{enumerate}
\item Numerical walls with respect to $u$ are either semicircles
with centers on the $\beta$-axis or rays parallel to the
$\alpha$-axis.

\item If two numerical walls intersect, then the two walls are
completely identical.
\item If $\mathcal{E}$ is
$\nu_{H,F}^{\alpha,\beta}$-semistable with
$\nu_{H,F}^{\alpha,\beta}(\mathcal{E})\neq+\infty$, then the object
$\mathcal{E}$ is $\nu_{H,F}^{\alpha,\beta}$-semistable along the
semicircle $\mathcal{C}_{\alpha,\beta}(\mathcal{E})$ in the
$(\alpha, \beta)$-plane with center $(0,
\beta+\nu_{H,F}^{\alpha,\beta}(\mathcal{E}))$ and radius
$\sqrt{\alpha^2+(\nu_{H,F}^{\alpha,\beta}(\mathcal{E}))^2}$.

\item If there is an exact sequence of $\nu_{H, F}^{\alpha,
\beta}$-semistable objects
$\mathcal{F}\hookrightarrow\mathcal{E}\twoheadrightarrow\mathcal{G}$
such that $\nu_{H, F}^{\alpha, \beta}(\mathcal{F})=\nu_{H,
F}^{\alpha, \beta}(\mathcal{G})$, then
$\overline{\Delta}_{H,F}(\mathcal{F})+\overline{\Delta}_{H,F}(\mathcal{G})\leq\overline{\Delta}_{H,F}(\mathcal{E})$.
Moreover, equality holds if and only if either $v(\mathcal{F})=0$,
$v(\mathcal{G})=0$, or both $\overline{\Delta}_{H,F}(\mathcal{E})=0$
and $v(\mathcal{F})$, $v(\mathcal{G})$ and $v(\mathcal{E})$ are all
proportional.
\item If $\overline{\Delta}_{H,F}(\mathcal{E})=0$ for a tilt
semistable object $\mathcal{E}$, then $\mathcal{E}$ can only be
destabilized at the unique numerical vertical wall.

\item Let $\mathcal{F}$ be a $\mu_{C}$-stable locally free sheaf on
$\mathcal{X}$ with $\overline{\Delta}_{H,F}(\mathcal{F})=0$. Then
$\mathcal{F}$ or $\mathcal{F}[1]$ is a $\nu_{H,F}^{\alpha,
\beta}$-stable object in $\Coh_C^{ \beta H}(\mathcal{X})$.

\item Let $\mathcal{F}$ be a $\mu_{H,F}$-stable torsion free sheaf on
$\mathcal{X}$. If $\mu^-_C(\mathcal{F})>\beta$, then
$\mathcal{F}\in\Coh_C^{\beta H}(\mathcal{X})$ and it is
$\nu_{H,F}^{\alpha,\beta}$-stable for $\alpha\gg0$.
\end{enumerate}
\end{proposition}

In order to reduce the relative tilt-stability to small $\alpha$, we
need the relative $\overline{\beta}$-stability which is a relative
version of $\overline{\beta}$-stability in \cite[Section 5]{BMS}.

\begin{definition}
For any $\mathcal{E}\in \Coh_C^{ \beta H}(\mathcal{X})$, we define
\begin{eqnarray*}
\overline{\beta}(\mathcal{E})= \left\{
\begin{array}{lcl}
\frac{HF\ch_1(\mathcal{E})-\sqrt{\overline{\Delta}_{H,F}(\mathcal{E})}}{H^2F\ch_0(\mathcal{E})},  & &\mbox{if}~\ch_0(\mathcal{E})\neq0,\\
&&\\
\frac{F\ch_2(\mathcal{E})}{HF\ch_1(\mathcal{E})}, &
&\mbox{otherwise}.
\end{array}\right.
\end{eqnarray*}
Moreover, we say that $\mathcal{E}$ is relative
$\overline{\beta}$-(semi)stable, if it is (semi)stable in a
neighborhood of $(0, \overline{\beta}(\mathcal{E}))$.
\end{definition}

By this definition we have
$F\ch^{\overline{\beta}(\mathcal{E})}_2(\mathcal{E})=0$.

\section{Conjectures and constructions}\label{S3}
In this section, we recall the definition of stability conditions on
triangulated category introduced by Bridgeland in \cite{Bri1} and
generalize \cite[Conjecture 5.2]{Sun2} to arbitrary relative
tilt-semistable objects with more precise form. Using this
conjecture, we then give a construction of stability conditions on
$\mathcal{X}$.

\subsection{Stability conditions}
Let $\mathcal{D}$ be a triangulated category, for which we fix a
finitely generated free abelian group $\Lambda$ and a group
homomorphism $v:\K(\mathcal{D})\rightarrow\Lambda$.
\begin{definition}\label{pre}
A stability condition on $\mathcal{D}$ is a pair $\sigma=(Z,
\mathcal{A})$, where $\mathcal{A}$ is the heart of a bounded
t-structure on $\mathcal{D}$, and $Z:\Lambda\rightarrow \mathbb{C}$
is a group homomorphism (called central charge) such that
\begin{enumerate}
\item $Z$ satisfies the following positivity property for any $0\neq\mathcal{E}\in
\mathcal{A}$:
$$Z(v(\mathcal{E}))\in\{re^{i\pi\phi}: r>0, 0<\phi\leq1\}.$$
\item $(Z, \mathcal{A})$ satisfies the Harder-Narasimhan
property: every object of $\mathcal{A}$ has a Harder-Narasimhan
filtration in $\mathcal{A}$ with respect to
$\nu_{\sigma}$-stability, here the slope $\nu_{\sigma}$ of an object
$\mathcal{E}\in \mathcal{A}$ is defined by
\begin{eqnarray*}
\nu_{\sigma}(\mathcal{E})= \left\{
\begin{array}{lcl}
+\infty,  & &\mbox{if}~\Im Z(v(\mathcal{E}))=0,\\
&&\\
-\frac{\Re Z(v(\mathcal{E}))}{\Im Z(v(\mathcal{E}))}, &
&\mbox{otherwise}.
\end{array}\right.
\end{eqnarray*}
\end{enumerate}
\end{definition}

We say $\mathcal{E}\in\mathcal{A}$ is $\nu_{\sigma}$-(semi)stable if
for any non-zero subobject $\mathcal{F}\subset \mathcal{E}$ in
$\mathcal{A}$, we have
$$\nu_{\sigma}(\mathcal{F})<(\leq)\nu_{\sigma}(\mathcal{E}/\mathcal{F}).$$
The Harder-Narasimhan filtration of an object $\mathcal{E}\in
\mathcal{A}$ is a chain of subobjects
$$0=\mathcal{E}_0\subset \mathcal{E}_1\subset\cdots\subset \mathcal{E}_m=\mathcal{E}$$ in $\mathcal{A}$ such
that $\mathcal{G}_i:=\mathcal{E}_i/\mathcal{E}_{i-1}$ is
$\nu_{\sigma}$-semistable and
$\nu_{\sigma}(\mathcal{G}_1)>\cdots>\nu_{\sigma}(\mathcal{G}_m)$. We
set $\nu_{\sigma}^+(\mathcal{E}):=\nu_{\sigma}(\mathcal{G}_1)$ and
$\nu_{\sigma}^-(\mathcal{E}):=\nu_{\sigma}(\mathcal{G}_m)$.

\begin{definition}\label{slicing}
For a stability condition $(Z, \mathcal{A})$ on $\mathcal{D}$ and
for $0<\phi\leq1$, we define the subcategory
$\mathcal{P}(\phi)\subset\mathcal{D}$ to be the category of
$\nu_{\sigma}$-semistable objects $\mathcal{E}\in\mathcal{A}$
satisfying $\tan(\pi\phi)=-1/\nu_{\sigma}(\mathcal{E})$. For other
$\phi\in \mathbb{R}$ the subcategory $\mathcal{P}(\phi)$ is defined
by the rule:
$$\mathcal{P}(\phi+1)=\mathcal{P}(\phi)[1].$$ The objects in $\mathcal{P}(\phi)$ is still called $\nu_{\sigma}$-semistable
objects.
\end{definition}

For an interval $I=(a,b)\subset\mathbb{R}$, we denote by
$\mathcal{P}(I)$ the extension-closure of
$$\bigcup_{\phi\in I}\mathcal{P}(\phi)\subset\mathcal{D}.$$ $\mathcal{P}(I)$ is a
quasi-abelian category when $b-a<1$ (cf. \cite[Definition
4.1]{Bri1}). If we have a distinguished triangle
$$A_1\xrightarrow{h} A_2\xrightarrow{g} A_3\rightarrow A_1[1]$$ with
$A_1,A_2,A_3\in \mathcal{P}(I)$, we say $h$ is a strict monomorphism
and $g$ is a strict epimorphism. Then we say that $\mathcal{P}(I)$
is of finite length if $\mathcal{P}(I)$ is Noetherian and Artinian
with respect to strict epimorphisms and strict monomorphisms,
respectively.

\begin{definition}
A stability condition $\sigma=(Z,\mathcal{A})$ is called locally
finite if there exists $\varepsilon>0$ such that for any $\phi\in
\mathbb{R}$, the quasi-abelian category
$\mathcal{P}((\phi-\varepsilon, \phi+\varepsilon))$ is of finite
length.
\end{definition}

\begin{definition}\label{supp}
We say a stability condition $\sigma=(Z,\mathcal{A})$ satisfies the
support property if there is a quadratic form $Q$ on $\Lambda$
satisfying $Q(v(\mathcal{E}))\geq0$ for any
$\nu_{\sigma}$-semistable object $\mathcal{E}\in\mathcal{A}$, and
$Q|_{\ker Z}$ is negative definite.
\end{definition}

\begin{remark}\label{remark3.5}
The local finiteness condition automatically follows if the support
property is satisfied (cf. \cite[Section 1.2]{KS} and \cite[Lemma
4.5]{Bri2}).
\end{remark}

\subsection{Conjectures}
We now generalize \cite[Conjecture 5.2]{Sun2} to arbitrary relative
tilt-semistable objects:
\begin{conjecture}\label{conj1}
Assume that $\mathcal{E}\in\Coh_C^{ \beta H}(\mathcal{X})$ is
$\nu^{\alpha,\beta}_{H, F}$-semistable for some $(\alpha,\beta)$ in
$\mathbb{R}_{>0}\times\mathbb{R}$ with $\nu^{\alpha,\beta}_{H,
F}(\mathcal{E})\neq+\infty$. Then
\begin{eqnarray}\label{3.1}
&&\left(F\ch_2^{\beta}(\mathcal{E})-\frac{\alpha^2}{2}H^2F\ch_0(\mathcal{E})\right)\left(H\ch_2^{\beta}(\mathcal{E})-\frac{H^3}{3H^2F}F\ch_2^{\beta}(\mathcal{E})\right)\\
\nonumber
&\geq&\left(\ch_3^{\beta}(\mathcal{E})-\frac{\alpha^2}{2}H^2\ch_1^{\beta}(\mathcal{E})+\frac{\alpha^2H^3}{3H^2F}HF\ch_1^{\beta}(\mathcal{E})\right)HF\ch_1^{\beta}(\mathcal{E}).
\end{eqnarray}
\end{conjecture}

We will show that Conjecture \ref{conj1} follows from a more natural
and seemingly weaker statement:
\begin{conjecture}\label{conj2}
Assume that $\mathcal{E}\in\Coh_C^{ \beta H}(\mathcal{X})$ is
$\nu^{\alpha,\beta}_{H, F}$-semistable for some $(\alpha,\beta)$ in
$\mathbb{R}_{>0}\times\mathbb{R}$ with $\nu^{\alpha,\beta}_{H,
F}(\mathcal{E})=0$. Then
\begin{equation}\label{3.2}
\ch_3^{\beta}(\mathcal{E})\leq\frac{\alpha^2}{2}H^2\ch_1^{\beta}(\mathcal{E})-\frac{\alpha^2H^3}{3H^2F}HF\ch_1^{\beta}(\mathcal{E}).
\end{equation}
\end{conjecture}

\begin{theorem}\label{thm3.8}
Conjecture \ref{conj1} holds if and only if Conjecture \ref{conj2}
holds for all $(\alpha,\beta)\in\mathbb{R}_{>0}\times\mathbb{R}$.
\end{theorem}
\begin{proof}
The proof is similar to that of \cite[Theorem 4.2]{BMS}. Consider
the below statement:
\begin{itemize}
\item[(*)] Assume that $\mathcal{E}$ is $\nu^{\alpha,\beta}_{H,
F}$-semistable $\nu^{\alpha,\beta}_{H, F}(\mathcal{E})\neq\infty$.
Let $\beta^{\prime}:=\beta+\nu^{\alpha,\beta}_{H, F}(\mathcal{E})$.
Then
\begin{equation}\label{3.3}
\ch_3^{\beta^{\prime}}(\mathcal{E})\leq(\alpha^2+\nu^{\alpha,\beta}_{H,
F}(\mathcal{E})^2)\left(\frac{1}{2}H^2\ch_1^{\beta^{\prime}}(\mathcal{E})-\frac{H^3}{3H^2F}HF\ch_1^{\beta^{\prime}}(\mathcal{E})\right).
\end{equation}
\end{itemize}
Obviously, Conjecture \ref{conj2} is a special case of (*).
Conversely, consider the assumptions of (*). By Proposition
\ref{Wall} (3), $\mathcal{E}$ is
$\nu^{\alpha^{\prime},\beta^{\prime}}_{H, F}$-semistable, where
$\alpha^{\prime 2}=\alpha^2+\nu^{\alpha,\beta}_{H,
F}(\mathcal{E})^2$. A simple computation shows
$\nu^{\alpha^{\prime},\beta^{\prime}}_{H, F}(\mathcal{E})=0$. Thus
Conjecture \ref{conj2} implies the statement (*).

Finally, a straightforward computation shows that the inequalities
(\ref{3.3}) and (\ref{3.1}) are equivalent. For this purpose, let us
use the abbreviations $\nu=\nu^{\alpha,\beta}_{H, F}(\mathcal{E})$
and $e_i=\ch_i^{\beta}(\mathcal{E})$ for $1\leq i\leq3$. Expanding
inequality (\ref{3.3}), one sees that
\begin{eqnarray*}
e_3-\nu
He_2+\frac{\nu^2H^2}{2}e_1-\frac{\nu^3H^3}{6}e_0&\leq&\frac{1}{2}(\alpha^2+\nu^2)(H^2e_1-\nu
H^3e_0)\\
&&-\frac{H^3}{3H^2F}(\alpha^2+\nu^2)(HFe_1-\nu H^2Fe_0).
\end{eqnarray*}
Collecting related terms, one obtains
$$e_3-\nu He_2\leq\frac{\alpha^2}{2}H^2e_1-\frac{H^3}{3H^2F}(\alpha^2+\nu^2)HFe_1-\frac{\alpha^2}{6}\nu
H^3e_0.$$ Substituting $\nu=(Fe_2-\frac{1}{2}\alpha^2H^2Fe_0)/HFe_1$
and multiplying with $HFe_1$ yields:
\begin{eqnarray*}
(HFe_1)e_3-(Fe_2-\frac{1}{2}\alpha^2H^2Fe_0)He_2&\leq&\frac{\alpha^2}{2}(H^2e_1)(HFe_1)-\frac{\alpha^2H^3}{3H^2F}(HFe_1)^2\\
&&-\frac{H^3}{3H^2F}(Fe_2-\frac{1}{2}\alpha^2H^2Fe_0)^2\\
&&-\frac{\alpha^2}{6}(Fe_2-\frac{1}{2}\alpha^2H^2Fe_0)H^3e_0.
\end{eqnarray*}
This simplifies to (\ref{3.1}).
\end{proof}

Considering the limit as $\alpha\rightarrow+\infty$, Conjecture
\ref{conj1} gives a Bogomolov type inequality:
\begin{proposition}\label{Bog-relative}
Let $\mathcal{E}$ be a $\mu_{H, F}$-semistable torsion free sheaf on
$\mathcal{X}$. Assume that Conjecture \ref{conj1} holds for all
$(\alpha, \beta)\in\mathbb{R}_{>0}\times\mathbb{R}$. Then
\begin{eqnarray}\label{nabla}
\nabla_{H,F}(\mathcal{E})&:=&\frac{1}{3}H^3\ch_0(\mathcal{E})F\ch_2(\mathcal{E})-\frac{2H^3}{3H^2F}(HF\ch_1(\mathcal{E}))^2\\
\nonumber&&+H^2\ch_1(\mathcal{E})HF\ch_1(\mathcal{E})-H^2F\ch_0(\mathcal{E})H\ch_2(\mathcal{E})\\
\nonumber&=&\widetilde{\Delta}_{H,F}(\mathcal{E})-\frac{H^3}{6H^2F}\overline{\Delta}_{H,F}(\mathcal{E})-\frac{H^3}{2H^2F}(HF\ch_1(\mathcal{E}))^2\\
\nonumber&\geq&0.
\end{eqnarray}
\end{proposition}
\begin{proof}
If $\mathcal{E}$ is not $\mu_{H,F}$-stable, we let $\mathcal{F}_1$,
$\cdots$, $\mathcal{F}_m$ be the stable factors of $\mathcal{E}$.
Then one has
$$\mu_{H,F}(\mathcal{F}_1)=\cdots=\mu_{H,F}(\mathcal{F}_m)=\mu_{H,F}(\mathcal{E})=\mu,$$
and thus
\begin{eqnarray*}
\frac{\nabla_{H,F}(\mathcal{E})}{H^2F\ch_0(\mathcal{E})}&=&\frac{H^3}{3H^2F}F\ch_2(\mathcal{E})-\frac{2\mu
H^3}{3H^2F}HF\ch_1(\mathcal{E}) +\mu
H^2\ch_1(\mathcal{E})-H\ch_2(\mathcal{E})\\
&=&\frac{\nabla_{H,F}(\mathcal{F}_1)}{H^2F\ch_0(\mathcal{F}_1)}+\cdots+\frac{\nabla_{H,F}(\mathcal{F}_m)}{H^2F\ch_0(\mathcal{F}_m)}.
\end{eqnarray*}
It follows that $\nabla_{H,F}(\mathcal{E})\geq0$, if
$\nabla_{H,F}(\mathcal{F}_1),\cdots,
\nabla_{H,F}(\mathcal{F}_m)\geq0$. Therefore we can reduce to the
case that $\mathcal{E}$ is $\mu_{H,F}$-stable.

By Proposition \ref{Wall}(7), one sees that $\mathcal{E}$ is
$\nu_{H,F}^{\alpha,\beta}$-stable for $\alpha\gg0$ and for some
$\beta\in\mathbb{R}$, and hence our assumption implies
\begin{eqnarray*}
&&\left(F\ch_2^{\beta}(\mathcal{E})-\frac{\alpha^2}{2}H^2F\ch_0(\mathcal{E})\right)\left(H\ch_2^{\beta}(\mathcal{E})-\frac{H^3}{3H^2F}F\ch_2^{\beta}(\mathcal{E})\right)\\
\nonumber
&\geq&\left(\ch_3^{\beta}(\mathcal{E})-\frac{\alpha^2}{2}H^2\ch_1^{\beta}(\mathcal{E})+\frac{\alpha^2H^3}{3H^2F}HF\ch_1^{\beta}(\mathcal{E})\right)HF\ch_1^{\beta}(\mathcal{E}).
\end{eqnarray*}
Taking $\alpha\rightarrow+\infty$, one deduces
\begin{eqnarray*}
&&-\frac{1}{2}H^2F\ch_0(\mathcal{E})\left(H\ch_2^{\beta}(\mathcal{E})-\frac{H^3}{3H^2F}F\ch_2^{\beta}(\mathcal{E})\right)\\
\nonumber
&\geq&\left(-\frac{1}{2}H^2\ch_1^{\beta}(\mathcal{E})+\frac{H^3}{3H^2F}HF\ch_1^{\beta}(\mathcal{E})\right)HF\ch_1^{\beta}(\mathcal{E}),
\end{eqnarray*}
i.e.,
\begin{eqnarray*}
\nabla^{\beta H}_{H,F}(\mathcal{E})&:=&\frac{1}{3}H^3\ch_0(\mathcal{E})F\ch^{\beta}_2(\mathcal{E})-\frac{2H^3}{3H^2F}(HF\ch^{\beta}_1(\mathcal{E}))^2\\
&&+H^2\ch^{\beta}_1(\mathcal{E})HF\ch^{\beta}_1(\mathcal{E})-H^2F\ch_0(\mathcal{E})H\ch^{\beta}_2(\mathcal{E})\\
&=&\widetilde{\Delta}^{\beta H}_{H,F}(\mathcal{E})-\frac{H^3}{6H^2F}\overline{\Delta}^{\beta H}_{H,F}(\mathcal{E})-\frac{H^3}{2H^2F}(HF\ch^{\beta}_1(\mathcal{E}))^2\\
&\geq&0.
\end{eqnarray*}
A straightforward calculation shows that $\nabla^{\beta
H}_{H,F}(\mathcal{E})=\nabla_{H,F}(\mathcal{E})$ for any
$\beta\in\mathbb{R}$. This finishes the proof.
\end{proof}

\begin{remark}
It seems that the inequality (\ref{nabla}) is highly non-trivial
even for line bundles. In fact, when $\mathcal{E}=\mathcal{O}(L)$
for some divisor $L$, the inequality (\ref{nabla}) becomes
$$
(H^2L)(HFL)-\frac{1}{2}(H^2F)(HL^2)
+\frac{1}{6}H^3(FL^2)-\frac{2H^3}{3H^2F}(HFL)^2\geq0.
$$
It seems not easy to prove.
\end{remark}

The following weaker inequality is more convenient to use to
construct stability conditions.
\begin{proposition}\label{pro3.11}
Assume that Conjecture \ref{conj1} holds for some $(\alpha,
\beta)\in\mathbb{R}_{>0}\times\mathbb{R}$. For any
$\nu^{\alpha,\beta}_{H, F}$-semistable object $\mathcal{E}$, we have
\begin{eqnarray}\label{3.5}
&&\left(F\ch_2^{\beta}(\mathcal{E})-\frac{\alpha^2}{2}H^2F\ch_0(\mathcal{E})\right)H\ch_2^{\beta}(\mathcal{E})\\
\nonumber
&\geq&\left(\ch_3^{\beta}(\mathcal{E})-\frac{\alpha^2}{2}H^2\ch_1^{\beta}(\mathcal{E})+\frac{\alpha^2H^3}{4H^2F}HF\ch_1^{\beta}(\mathcal{E})\right)HF\ch_1^{\beta}(\mathcal{E}).
\end{eqnarray}
\end{proposition}
\begin{proof}
When $HF\ch_1^{\beta}(\mathcal{E})=0$, the inequality (\ref{3.5})
immediately follows from Lemma \ref{ch2}. By Theorem \ref{Bog-tilt},
one sees that
\begin{eqnarray*}
&&\left(F\ch_2^{\beta}(\mathcal{E})-\frac{\alpha^2}{2}H^2F\ch_0(\mathcal{E})\right)\left(H\ch_2^{\beta}(\mathcal{E})-\frac{H^3}{3H^2F}F\ch_2^{\beta}(\mathcal{E})\right)\\
&=&\left(F\ch_2^{\beta}(\mathcal{E})-\frac{\alpha^2}{2}H^2F\ch_0(\mathcal{E})\right)H\ch_2^{\beta}(\mathcal{E})
-\frac{H^3}{3H^2F}(F\ch_2^{\beta}(\mathcal{E}))^2\\
&&+\frac{\alpha^2H^3}{6H^2F}H^2F\ch_0(\mathcal{E})F\ch_2^{\beta}(\mathcal{E})\\
&\leq&\left(F\ch_2^{\beta}(\mathcal{E})-\frac{\alpha^2}{2}H^2F\ch_0(\mathcal{E})\right)H\ch_2^{\beta}(\mathcal{E})+\frac{\alpha^2H^3}{12H^2F}(HF\ch_1^{\beta}(\mathcal{E}))^2.
\end{eqnarray*}
Therefore, if $HF\ch_1^{\beta}(\mathcal{E})\neq0$, our assumption
implies
\begin{eqnarray*}
&&\left(\ch_3^{\beta}(\mathcal{E})-\frac{\alpha^2}{2}H^2\ch_1^{\beta}(\mathcal{E})+\frac{\alpha^2H^3}{3H^2F}HF\ch_1^{\beta}(\mathcal{E})\right)HF\ch_1^{\beta}(\mathcal{E})\\
&\leq&\left(F\ch_2^{\beta}(\mathcal{E})-\frac{\alpha^2}{2}H^2F\ch_0(\mathcal{E})\right)H\ch_2^{\beta}(\mathcal{E})+\frac{\alpha^2H^3}{12H^2F}(HF\ch_1^{\beta}(\mathcal{E}))^2.
\end{eqnarray*}
This simplifies to (\ref{3.5}).
\end{proof}

In order to apply Conjecture \ref{conj1} to construct stability
conditions on $\mathcal{X}$, we now introduce the mixed
tilt-stability with a slight modification of the notion in
\cite[Section 4]{Sun2}. Let $t$ be a positive rational number. For
any $\mathcal{E}\in \Coh_C^{ \beta H}(\mathcal{X})$, its mixed
tilt-slope $\nu_{\alpha, \beta,t}$ is defined by
\begin{eqnarray*}
\nu_{\alpha, \beta, t}(\mathcal{E})= \left\{
\begin{array}{lcl}
+\infty,  & &\mbox{if}~FH\ch^{\beta}_1(\mathcal{E})=0,\\
&&\\
\frac{(H+tF)\ch_2^{\beta}(\mathcal{E})-\frac{t}{2}\alpha^2FH^{2}\ch^{\beta}_0(\mathcal{E})}{HF\ch^{\beta}_1(\mathcal{E})},
& &\mbox{otherwise}.
\end{array}\right.
\end{eqnarray*}
The $\nu_{\alpha,\beta,t}$-stability of $\mathcal{E}$ is defined as
before in Definition \ref{slope}. By \cite[Theorem 4.1]{Sun2},
$\nu_{\alpha,\beta,t}$-stability gives a notion of stability when
$(\alpha,
\beta,t)\in\sqrt{\mathbb{Q}_{>0}}\times\mathbb{Q}\times\mathbb{Q}_{>0}$,
in the sense that it satisfies the weak see-saw property and the
Harder-Narasimhan property. We also call it mixed tilt-stability.

%\begin{theorem}\label{Bog-mixed}
%Let $\mathcal{E}\in\Coh_C^{\beta H}(\mathcal{X})$ be a
%$\nu_{\alpha,\beta, t}$-semistable object. Then we have
%$$
%\widetilde{\Delta}^{\beta H}_{H,F,t}(\mathcal{E}):=
%\widetilde{\Delta}^{\beta H}_{H,F}(\mathcal{E})+
%\frac{t}{2}\overline{\Delta}^{\beta
%H}_{H,F}(\mathcal{E})+\frac{t}{2}(FH\ch_1^{\beta}(\mathcal{E}))^2\geq0.
%$$
%\end{theorem}
%\begin{proof}
%The proof is the same as that of \cite[Theorem 4.3]{Sun2}.
%\end{proof}

\begin{theorem}\label{thm3.9}
Assume that Conjecture \ref{conj1} holds for some $(\alpha,
\beta)\in\sqrt{\mathbb{Q}_{>0}}\times\mathbb{Q}$ and $\mathcal{E}$
is $\nu_{\alpha,\beta,t}$-semistable. Then
\begin{eqnarray}\label{3.6}
&&\left(F\ch_2^{\beta}(\mathcal{E})-\frac{\alpha^2}{2}H^2F\ch_0(\mathcal{E})\right)H\ch_2^{\beta}(\mathcal{E})\\
\nonumber
&\geq&\left(\ch_3^{\beta}(\mathcal{E})-\frac{\alpha^2}{2}H^2\ch_1^{\beta}(\mathcal{E})+\frac{\alpha^2H^3}{4H^2F}HF\ch_1^{\beta}(\mathcal{E})\right)HF\ch_1^{\beta}(\mathcal{E}).
\end{eqnarray}
Moreover, we have
$\ch_3^{\beta}(\mathcal{E})\leq\frac{\alpha^2}{2}H^2\ch_1^{\beta}(\mathcal{E})-\frac{\alpha^2H^3}{4H^2F}HF\ch_1^{\beta}(\mathcal{E})$
if $\nu_{\alpha,\beta,t}(\mathcal{E})=0$.
\end{theorem}
\begin{proof}
We use the notations by Liu in \cite{Liu}:
\begin{eqnarray*}
a(\mathcal{E})&=&-F\ch_2^{\beta}(\mathcal{E})+\frac{\alpha^2}{2}H^2F\ch_0(\mathcal{E})\\
b(\mathcal{E})&=&-\ch_3^{\beta}(\mathcal{E})+\frac{\alpha^2}{2}H^2\ch_1^{\beta}(\mathcal{E})-\frac{\alpha^2H^3}{4H^2F}HF\ch_1^{\beta}(\mathcal{E})\\
c(\mathcal{E})&=&HF\ch_1^{\beta}(\mathcal{E})\\
d(\mathcal{E})&=&H\ch_2^{\beta}(\mathcal{E}).
\end{eqnarray*}
Then inequality (\ref{3.5}) and (\ref{3.6}) become
$b(\mathcal{E})c(\mathcal{E})\geq a(\mathcal{E})d(\mathcal{E})$ and
$\nu_{\alpha,\beta,t}(\mathcal{E})=\frac{d(\mathcal{E})-ta(\mathcal{E})}{c(\mathcal{E})}$.

If $c(\mathcal{E})=0$, by Lemma \ref{ch2}, one sees that
$a(\mathcal{E})\leq0$ and $d(\mathcal{E})\geq0$. Hence
$b(\mathcal{E})c(\mathcal{E})\geq a(\mathcal{E})d(\mathcal{E})$ in
this case. The proof of the case of $c(\mathcal{E})>0$ is the same
as that of \cite[Theorem 5.4]{Liu}.

If $\nu_{\alpha,\beta,t}(\mathcal{E})=0$, one sees that
$d(\mathcal{E})=ta(\mathcal{E})$ and $c(\mathcal{E})>0$. Hence the
inequality $b(\mathcal{E})c(\mathcal{E})\geq
a(\mathcal{E})d(\mathcal{E})$ implies $b(\mathcal{E})\geq0$. This
completes the proof.
\end{proof}

\subsection{Constructions of stability conditions}
We give the construction of the heart
$\mathcal{A}_{t}^{\alpha,\beta}(\mathcal{X})$ of a bounded
$t$-structure on $\D^b(\mathcal{X})$ as a tilt starting from
$\Coh_C^{\beta H}(\mathcal{X})$. We consider the torsion pair
$(\mathcal{T}_t^{\prime},\mathcal{F}_t^{\prime})$ in $\Coh_C^{\beta
H}(\mathcal{X})$ as follows:
\begin{eqnarray*}
\mathcal{T}_t^{\prime}&=&\{\mathcal{E}\in\Coh_C^{\beta
H}(\mathcal{X}): \text{any quotient}~
\mathcal{E}\twoheadrightarrow\mathcal{G}~ \text{satisfies}~
\nu_{\alpha,\beta,t}(\mathcal{G})>0 \}\\
\mathcal{F}_t^{\prime}&=&\{\mathcal{E}\in\Coh_C^{\beta
H}(\mathcal{X}):\text{any subobject}~
\mathcal{K}\hookrightarrow\mathcal{E}~ \text{satisfies}~
\nu_{\alpha,\beta,t}(\mathcal{K})\leq0 \}.
\end{eqnarray*}

\begin{definition}
We define the abelian category
$\mathcal{A}_{t}^{\alpha,\beta}(\mathcal{X})\subset
\D^b(\mathcal{X})$ to be the extension-closure
$$\mathcal{A}_{t}^{\alpha,\beta}(\mathcal{X})=\langle\mathcal{T}_t^{\prime},
\mathcal{F}_t^{\prime}[1]\rangle.$$
\end{definition}
For $s,t\in\mathbb{Q}_{>0}$, consider the following central charge
$$z_{s,t}=
(s-\frac{\alpha^2H^3}{4H^2F})HF\ch_1^{\beta}-\ch_3^{\beta}+\frac{\alpha^2}{2}H^2\ch_1^{\beta}+i\Big(H\ch_2^{\beta}+tF\ch_2^{\beta}-\frac{t\alpha^2
}{2}H^2F\ch_0^{\beta}\Big).$$ We think of it as the composition
$$z_{s,t}: \K(\D^b(\mathcal{X}))\xrightarrow{\bar{v}} \mathbb{Z}\oplus\mathbb{Z}\oplus\mathbb{Z}
\oplus\frac{1}{2}\mathbb{Z}\oplus\frac{1}{2}\mathbb{Z}\oplus\frac{1}{6}\mathbb{Z}
\xrightarrow{Z_{s,t}} \mathbb{C},$$ where the first map is given by
$$\bar{v}(\mathcal{E})=\left(\ch_0(\mathcal{E}), FH\ch_1(\mathcal{E}), H^2\ch_1(\mathcal{E}), F\ch_2(\mathcal{E}), H\ch_2(\mathcal{E}),\ch_3(\mathcal{E})\right).$$

\begin{theorem}\label{main}
Assume that Conjecture \ref{conj1} holds for some $(\alpha,
\beta)\in\sqrt{\mathbb{Q}_{>0}}\times\mathbb{Q}$ and the Bogomolov
inequality holds on every fiber of $f$, then the pair
$(Z_{s,t},\mathcal{A}_{t}^{\alpha, \beta}(\mathcal{X}))$ is a
stability condition on $\mathcal{X}$ satisfying the support property
for $s,t\in\mathbb{Q}_{>0}$.
\end{theorem}
\begin{proof}
The proof of the positivity property and the Harder-Narasimhan
property is the same as that of \cite[Theorem 5.3]{Sun2}. Indeed,
the key ingredients in the proof of \cite[Theorem 5.3]{Sun2} are
Proposition \ref{Noether}, Lemma \ref{ch2}, \cite[Theorem 4.3]{Sun2}
and \cite[Proposition 4.7]{Sun2} which still hold in our case. The
support property follows from Theorem \ref{thm3.9} and \cite[Lemma
5.6]{Liu}.
\end{proof}

\section{Stability conditions on ruled threefolds}\label{S4}
Throughout this section we let $E$ be a locally free sheaf on $C$
with $\rank E=3$ and $\mathcal{X}:=\mathbb{P}(E)$ be the projective
bundle associated to $E$ with the projection
$f:\mathcal{X}\rightarrow C$ and the associated relative ample
invertible sheaf $\mathcal{O}_{\mathcal{X}}(1)$. Since
$\mathbb{P}(E)\cong \mathbb{P}(E\otimes L)$ for any line bundle $L$
on $C$, we can assume that $H:=c_1(\mathcal{O}_{\mathcal{X}}(1))$ is
nef. One sees that $H^3=\deg E$ and $H^2F=1$. We freely use the
notations in previous sections.

\begin{lemma}\label{can}
Denote by $g$ the genus of $C$. Then we have
\begin{eqnarray*}
c_1(T_{\mathcal{X}})&=&-f^*K_C-f^*c_1(E)+3H\\
c_2(T_{\mathcal{X}})&=& 3H^2-(6g-6+2\deg E)HF.
\end{eqnarray*}
In particular, for any divisor $D$ on $\mathcal{X}$ we have
\begin{eqnarray*}
DFc_1(T_{\mathcal{X}})&=&3DHF\\
DHc_1(T_{\mathcal{X}})&=&3DH^2-(2g-2+H^3)DHF\\
D\Big(c^2_1(T_{\mathcal{X}})+c_2(T_{\mathcal{X}})\Big)&=&
12DH^2-(18g-18+8H^3)DHF\\
\chi(\mathcal{O}_{\mathcal{X}})&=&\frac{1}{24}c_1(T_{\mathcal{X}})c_2(T_{\mathcal{X}})=1-g.
\end{eqnarray*}
\end{lemma}
\begin{proof}
See \cite[Lemma 6.1]{Sun2}
\end{proof}

\begin{proposition}\label{bar}
Let $\mathcal{E}$ be a relative $\overline{\beta}$-semistable object
on $\mathcal{X}=\mathbb{P}(E)$ with
$HF\ch^{\beta_0}_1(\mathcal{E})\neq0$ for some
$\beta_0\in\mathbb{R}$. Then we have
$$
HF\ch_1(\mathcal{E})=F\ch_2(\mathcal{E})=H\ch_2(\mathcal{E})=0
~\mbox{and}~\ch_3(\mathcal{E})\leq0.
$$
\end{proposition}
\begin{proof}
Tensoring with lines bundles $\mathcal{O}_{\mathcal{X}}(aH)$ for
$a\in \mathbb{Z}$ make it possible to reduce to the case of
$0\leq\overline{\beta}(\mathcal{E})<1$. If
$\overline{\Delta}_{H,F}(\mathcal{E})=0$, from
$HF\ch_1^{\beta_0}(\mathcal{E})\neq0$, it follows that
$\ch_0(\mathcal{E})\neq0$. This implies that
\begin{eqnarray*}
\nu_{H,F}^{\alpha,
\beta}(\mathcal{E})&=&\frac{F\ch^{\beta}_2(\mathcal{E})-\frac{\alpha^2}{2}\ch_0(\mathcal{E})}{HF\ch^{\beta}_1(\mathcal{E})}\\
&=&\frac{ (\overline{\beta}(\mathcal{E})-\beta)^2-\alpha^2}{
2(\overline{\beta}(\mathcal{E})-\beta)}\\
&\rightarrow&0
\end{eqnarray*}
for $(\alpha,\beta)\rightarrow (0, \overline{\beta}(\mathcal{E}))$.
Since $HF\ch_1^{\overline{\beta}(\mathcal{E})}(\mathcal{E})>0$ when
$\overline{\Delta}_{H,F}(\mathcal{E})>0$, we have
$\nu_{H,F}^{\alpha, \beta}(\mathcal{E})\rightarrow0$ for
$(\alpha,\beta)\rightarrow (0, \overline{\beta}(\mathcal{E}))$ in
any case.

Let $n$ be an integer satisfying $1\leq n\leq2$. By Proposition
\ref{Wall}(6), one sees that $\mathcal{O}_{\mathcal{X}}(nH)$ and
$\mathcal{O}_{\mathcal{X}}(K_{\mathcal{X}}+nH)[1]$ are
$\nu_{H,F}^{\alpha, \beta}$-stable for all $\alpha>0$ and
$0\leq\beta<1$. For $(\alpha,\beta)\rightarrow (0,
\overline{\beta}(\mathcal{E}))$, we have
\begin{eqnarray*}
\nu_{H,F}^{\alpha,
\beta}(\mathcal{O}_{\mathcal{X}}(nH))&\rightarrow&\frac{n-\overline{\beta}(\mathcal{E})}{2}>0\\
\nu_{H,F}^{\alpha,
\beta}(\mathcal{O}_{\mathcal{X}}(K_{\mathcal{X}}+nH)[1])&\rightarrow&\frac{n-3-\overline{\beta}(\mathcal{E})}{2}<0,
\end{eqnarray*}
and therefore $$\nu_{H,F}^{\alpha,
\beta}(\mathcal{O}_{\mathcal{X}}(nH))>\nu_{H,F}^{\alpha,
\beta}(\mathcal{E})>\nu_{H,F}^{\alpha,
\beta}(\mathcal{O}_{\mathcal{X}}(K_{\mathcal{X}}+nH)).$$ Applying
the standard Hom-vanishing between stable objects and Serre duality,
we conclude $$\Hom(\mathcal{O}_{\mathcal{X}}(nH),
\mathcal{E})=0~\mbox{and}~\Ext^2(\mathcal{O}_{\mathcal{X}}(nH),
\mathcal{E})=0.$$ The Hirzebruch-Riemann-Roch Theorem implies
\begin{eqnarray*}
0&\geq&\chi(\mathcal{O}_{\mathcal{X}}(H),
\mathcal{E})\\
&=&\ch^{H}_3(\mathcal{E})+\frac{c_1(\mathcal{X})}{2}\ch^{H}_2(\mathcal{E})+\frac{c_1^2(\mathcal{X})
+c_2(\mathcal{X})}{12}\ch^{H}_1(\mathcal{E})+\chi(\mathcal{O}_{\mathcal{X}})\ch_0(\mathcal{E})\\
&=&\ch^{H}_3(\mathcal{E})+\Big(3H-(2g-2+H^3)F\Big)\ch^{H}_2(\mathcal{E})+(1-g)\ch_0(\mathcal{E})\\
&&+\Big(H^2-(\frac{3}{2}g-\frac{3}{2}+\frac{2}{3})HF\Big)\ch^{H}_1(\mathcal{E})\\
&=&\ch_3(\mathcal{E})+\frac{1}{2}H\ch_2(\mathcal{E})-(g-1+\frac{H^3}{2})F\ch_2(\mathcal{E})-(\frac{1}{2}g-\frac{1}{2}+\frac{1}{6}H^3)HF\ch_1(\mathcal{E})
\end{eqnarray*}
and
\begin{eqnarray*}
0&\geq&\chi(\mathcal{O}_{\mathcal{X}}(2H),
\mathcal{E})\\
&=&\ch^{2H}_3(\mathcal{E})+\Big(3H-(2g-2+H^3)F\Big)\ch^{2H}_2(\mathcal{E})+(1-g)\ch_0(\mathcal{E})\\
&&+\Big(H^2-(\frac{3}{2}g-\frac{3}{2}+\frac{2}{3})HF\Big)\ch^{2H}_1(\mathcal{E})\\
&=&\ch_3(\mathcal{E})-\frac{1}{2}H\ch_2(\mathcal{E})-(g-1+\frac{H^3}{2})F\ch_2(\mathcal{E})+(\frac{1}{2}g-\frac{1}{2}+\frac{1}{3}H^3)HF\ch_1(\mathcal{E}).
\end{eqnarray*}
Thus one obtains
$$\ch_3(\mathcal{E})\leq-\frac{1}{2}H\ch_2(\mathcal{E})+(g-1+\frac{H^3}{2})F\ch_2(\mathcal{E})+(\frac{1}{2}g-\frac{1}{2}+\frac{1}{6}H^3)HF\ch_1(\mathcal{E})$$
and
$$\ch_3(\mathcal{E})\leq\frac{1}{2}H\ch_2(\mathcal{E})+(g-1+\frac{H^3}{2})F\ch_2(\mathcal{E})-(\frac{1}{2}g-\frac{1}{2}+\frac{1}{3}H^3)HF\ch_1(\mathcal{E}).$$
By Lemma \ref{tensor}, one sees that the above two inequalities also
hold for $\mathcal{E}(aF)$ for any $a\in\mathbb{Z}$. This implies
that
$$F\ch_2(\mathcal{E})=-\frac{1}{2}HF\ch_1(\mathcal{E})=\frac{1}{2}HF\ch_1(\mathcal{E}).$$
Therefore one gets
\begin{equation*}
F\ch_2(\mathcal{E})=HF\ch_1(\mathcal{E})=0.
\end{equation*}
It follows that
$\ch_3(\mathcal{E})\leq-\frac{1}{2}H\ch_2(\mathcal{E})$ and
$\ch_3(\mathcal{E})\leq\frac{1}{2}H\ch_2(\mathcal{E})$. They imply
\begin{equation*}
\ch_3(\mathcal{E})\leq0 ~\mbox{and}~H\ch_2(\mathcal{E})=0.
\end{equation*}
This completes the proof.
\end{proof}

We show that Conjecture \ref{conj2} holds for
$\mathcal{X}=\mathbb{P}(E)$ and any $(\alpha,\beta)\in
\mathbb{R}_{>0}\times\mathbb{R}$.

\begin{theorem}\label{main2}
Let $\mathcal{E}$ be a $\nu^{\alpha,\beta}_{H, F}$-semistable object
on $\mathcal{X}=\mathbb{P}(E)$ with $\nu^{\alpha,\beta}_{H,
F}(\mathcal{E})=0$. Then we have
$$
\ch_3^{\beta}(\mathcal{E})\leq\frac{\alpha^2}{2}H^2\ch_1^{\beta}(\mathcal{E})-\frac{\alpha^2H^3}{3H^2F}HF\ch_1^{\beta}(\mathcal{E}).
$$
\end{theorem}
\begin{proof}
We proceed by induction on $\overline{\Delta}_{H,F}(\mathcal{E})$,
which by Theorem \ref{Bog-tilt} is a non-negative integer-valued
function on objects of $\Coh_C^{\beta H}(\mathcal{X})$.

If $\overline{\Delta}_{H,F}(\mathcal{E})=0$, then $\mathcal{E}$ is
relative $\overline{\beta}$-semistable by Proposition \ref{Wall}(5).
From Proposition \ref{bar}, one sees that
\begin{equation}\label{4.1}
HF\ch_1(\mathcal{E})=F\ch_2(\mathcal{E})=H\ch_2(\mathcal{E})=0~\mbox{and}~\ch_3(\mathcal{E})\leq0,
\end{equation}
and hence
\begin{eqnarray}\label{4.2}
\ch_3^{\beta}(\mathcal{E})&=&\ch_3(\mathcal{E})-\beta
H\ch_2(\mathcal{E})+\frac{\beta^2}{2}H^2\ch_1(\mathcal{E})-\frac{\beta^3}{6}H^3\ch_0(\mathcal{E})\\
\nonumber&\leq&\frac{\beta^2}{2}H^2\ch_1(\mathcal{E})-\frac{\beta^3}{6}H^3\ch_0(\mathcal{E})\\
\nonumber&=&\frac{\beta^2}{2}H^2\ch^{\beta}_1(\mathcal{E})+\frac{\beta^3}{2}H^3\ch_0(\mathcal{E})-\frac{\beta^3}{6}H^3\ch_0(\mathcal{E})\\
\nonumber&=&\frac{\beta^2}{2}H^2\ch^{\beta}_1(\mathcal{E})+\frac{\beta^3}{3}H^3\ch_0(\mathcal{E}).
\end{eqnarray}
On the other hand, since $\nu^{\alpha,\beta}_{H, F}(\mathcal{E})=0$,
one infers
$$0<HF\ch^{\beta}_1(\mathcal{E})=HF\ch_1(\mathcal{E})-\beta H^2F\ch_0(\mathcal{E})=-\beta H^2F\ch_0(\mathcal{E})$$
and
\begin{eqnarray*}
\frac{\alpha^2}{2}H^2F\ch_0(\mathcal{E})&=&F\ch_2^{\beta}(\mathcal{E})\\
&=&F\ch_2(\mathcal{E})-\beta
HF\ch_1(\mathcal{E})+\frac{\beta^2}{2}H^2F\ch_0(\mathcal{E})\\
&=&\frac{\beta^2}{2}H^2F\ch_0(\mathcal{E}).
\end{eqnarray*}
Thus $\ch_0(\mathcal{E})\neq0$ and $\beta^2=\alpha^2$. Equality
(\ref{4.1}) implies that
$\beta=-\frac{HF\ch_1^{\beta}(\mathcal{E})}{H^2F\ch_0(\mathcal{E})}$.
Therefore one deduces
$$\beta^3=\alpha^2\beta=-\alpha^2\frac{HF\ch_1^{\beta}(\mathcal{E})}{H^2F\ch_0(\mathcal{E})}.$$
Substituting it into (\ref{4.2}), one concludes that
$$
\ch_3^{\beta}(\mathcal{E})\leq\frac{\alpha^2}{2}H^2\ch_1^{\beta}(\mathcal{E})-\frac{\alpha^2H^3}{3H^2F}HF\ch_1^{\beta}(\mathcal{E}).
$$

Now we assume that $\overline{\Delta}_{H,F}(\mathcal{E})>0$. One
sees that $\mathcal{E}$ is not relative
$\overline{\beta}$-semistable. Otherwise we obtain the conclusion of
Proposition \ref{bar} which contradicts
$\overline{\Delta}_{H,F}(\mathcal{E})>0$. Hence $\mathcal{E}$ is
destabilized along a wall between $(\alpha, \beta)$ and
$(0,\overline{\beta}(\mathcal{E}))$. Let $\mathcal{F}_1$, $\cdots$,
$\mathcal{F}_m$ be the stable factors of $\mathcal{E}$ along this
wall. By induction, the desired inequality holds for
$\mathcal{F}_1$, $\cdots$, $\mathcal{F}_m$ and so it does for
$\mathcal{E}$ by the linearity of Chern character.
\end{proof}

\begin{corollary}\label{cor1}
Let $\mathcal{E}$ be a $\mu_{H, F}$-semistable torsion free sheaf on
$\mathcal{X}=\mathbb{P}(E)$. Then
$$\widetilde{\Delta}_{H,F}(\mathcal{E})\geq\frac{H^3}{6H^2F}\overline{\Delta}_{H,F}(\mathcal{E})+\frac{H^3}{2H^2F}(HF\ch_1(\mathcal{E}))^2.$$
\end{corollary}
\begin{proof}
The conclusion follows from Proposition \ref{Bog-relative} and
Theorem \ref{main2}.
\end{proof}

\begin{corollary}\label{cor2}
There exist stability conditions satisfying the support property on
$\mathcal{X}=\mathbb{P}(E)$.
\end{corollary}
\begin{proof}
The conclusion follows from Theorem \ref{thm3.8}, Theorem
\ref{main2} and Theorem \ref{main}.
\end{proof}

\bibliographystyle{amsplain}

\begin{thebibliography}{10}
\bibitem{BLMNPS}
A. Bayer, M. Lahoz, E. Macr\`i, H. Nuer, A. Perry and P. Stellari,
Stability conditions in families. Publ. Math. Inst. Hautes \'Etudes
Sci. 133 (2021), 157--325.

%\bibitem{BLMS}
%A. Bayer, M. Lahoz, E. Macr\`i and P. Stellari, Stability conditions
%on Kuznetsov components, 2017, Appendix about the Torelli theorem
%for cubic fourfolds by A. Bayer, M. Lahoz, E. Macr\`i, P. Stellari,
%and X. Zhao, arXiv:1703.10839.

\bibitem{BMS} A. Bayer, E. Macr\`i and P. Stellari, Stability conditions on
abelian threefolds and some Calabi-Yau threefolds. Invent. Math. 206
(2016), no. 3, 869--933.

\bibitem{BMT}A. Bayer, E. Macr\`i and Y. Toda, Bridgeland stability
conditions on threefolds I: Bogomolov-Gieseker type inequalities. J.
Algebraic Geom. 23 (2014), 117--163.


\bibitem{BMSZ}M. Bernardara, E. Macr\`i, B. Schmidt and X. Zhao,
Bridgeland Stability Conditions on Fano Threefolds. \'Epijournal
Geom. Alg\'ebrique 1 (2017), Art. 2, 24 pp.

\bibitem{Bri1}T. Bridgeland, Stability conditions on triangulated categories. Ann. of Math.
166 (2007), no. 2, 317--345.

\bibitem{Bri2}T. Bridgeland, Stability conditions on K3 surfaces. Duke Math. J.
141 (2008), no. 2, 241--291.


\bibitem{HRS}
D. Happel, I. Reiten, and S. Smal\o, Tilting in abelian categories
and quasitilted algebras. Mem. Amer. Math. Soc. 120 (1996), viii+
88.


%\bibitem{Hart}R. Hartshorne, Algebraic Geometry. Graduate Texts in Mathematics, Springer-Verlag,
%1977.



\bibitem{KS}
M. Kontsevich and Y. Soibelman, Stability structures, motivic
Donaldson-Thomas invariants and cluster transformations.
arXiv:0811.2435.

\bibitem{Kos}
N. Koseki, Stability conditions on product threefolds of projective
spaces and abelian varieties. Bull. Lond. Math. Soc. 50 (2017), no.
2, 229--244.


%\bibitem{Kos2}
%N. Koseki, Stability conditions on threefolds with nef tangent
%bundles. Adv. Math. 372 (2020), 107--316.

\bibitem{Langer1}A. Langer, Semistable sheaves in positive characteristic. Ann. Math. 159 (2004), 241--276.

%\bibitem{Li1}C. Li, Stability conditions on Fano threefolds of Picard number 1. J. Eur. Math. Soc. 21 (2019), no. 3, 709--726.


\bibitem{Li2}C. Li, On stability conditions for the quintic
threefold. Invent. Math. 218 (2019), 301--340.

\bibitem{Liu}Y. Liu, Stability conditions on product varieties.
J. Reine Angew. Math. 770 (2021), 135--157.

\bibitem{Maci}
A. Maciocia, Computing the walls associated to Bridgeland stability
conditions on projective surfaces. Asian J. Math. 18 (2014), no. 2,
263--279.

%\bibitem{MP1}
%A. Maciocia and D. Piyaratne, Fourier-Mukai transforms and
%Bridgeland stability conditions on Abelian threefolds II. Int. J.
%Math. 27 (2013), no. 1, 1650007, 27.

%\bibitem{MP2}
%A. Maciocia and D. Piyaratne, Fourier-Mukai transforms and
%Bridgeland stability conditions on abelian threefolds. Algebraic
%Geom. 2 (2015), no. 3, 270--297.

%\bibitem{Mac2}E. Macr\`i, A generalized Bogomolov-Gieseker inequality for the
%three-dimensional projective space. Algebra Number Theory 8 (2014),
%no. 1, 173--190.



%\bibitem{MS}E. Macr\`i and B. Schmidt, Lectures on Bridgeland
%Stability. In: Moduli of Curves. Lecture Notes of the Unione
%Matematica Italiana, vol 21, Springer 2017.


\bibitem{Piy}D. Piyaratne, Stability conditions, Bogomolov-Gieseker type inequalities and Fano 3-folds. arXiv:1705.04011.


%\bibitem{PT}D. Piyaratne and Y. Toda, Moduli of Bridgeland semistable objects on 3-folds and
%Donaldson-Thomas invariants. J. reine angew. Math. 747 (2019),
%175--219.

%\bibitem{Sch1}B. Schmidt, A generalized Bogomolov-Gieseker inequality for the
%smooth quadric threefold. Bull. Lond. Math. Soc. 46 (2014), no. 5,
%915--923.

%\bibitem{Sch2}B. Schmidt, Counterexample to the generalized Bogomolov-Gieseker
%inequality for threefolds. Int. Math. Res. Not. 8 (2017),
%2562--2566.

\bibitem{Sun}H. Sun, Stability conditions on threefolds with vanishing Chern classes. arXiv:2006.00756.

\bibitem{Sun2}H. Sun, Stability conditions on fibred threefolds. arXiv: 2201.13251

%\bibitem{Toda1}Y. Toda, Curve counting theories via stable objects I. DT/PT correspondence. J. Amer. Math. Soc.
%23 (2010), no. 4, 1119--1157.
\end{thebibliography}

\end{document}